\documentclass[11pt]{amsart}
\usepackage{amsmath,amsfonts,amsbsy,amsgen,amscd,mathrsfs,amssymb,amsthm}
\usepackage{enumerate,mathtools}
\usepackage{cases}
\usepackage{dsfont}

\usepackage{clipboard}
\newclipboard{myclipboard}

\usepackage[T1]{fontenc} 

\usepackage{fullpage}
\usepackage{url}

\usepackage{mathrsfs}

\usepackage[colorlinks=true,linkcolor=blue]{hyperref} 
\makeatletter
\renewcommand*{\eqref}[1]{%
  \hyperref[{#1}]{\textup{\tagform@{\ref*{#1}}}}%
}
\makeatother

\usepackage{tikz}
\usetikzlibrary{calc}
\usetikzlibrary{intersections}
\usetikzlibrary{positioning}
\usetikzlibrary{hobby}

\usetikzlibrary{shadings}

\usetikzlibrary{decorations}
\numberwithin{figure}{section}

\numberwithin{equation}{section}

\newtheorem{theorem}{Theorem}[section]
\newtheorem*{theorem*}{Theorem}
\newtheorem{proposition}[theorem]{Proposition}

\newtheorem*{corollary*}{Corollary}
\newtheorem{lemma}[theorem]{Lemma}
\theoremstyle{definition}

\newtheorem{remark}[theorem]{Remark}

\newtheorem*{claim*}{Claim}

\newcommand{\R}{\mathbb R}
\newcommand{\N}{\mathbb N}

\newcommand{\PP}{\mathbb P}
\newcommand{\EE}{\mathbb E}
\newcommand{\Ent}{\operatorname{Ent}}

\newcommand{\ins}{\lambda}
\newcommand{\TT}{T}
\newcommand*\diff{\mathop{}\!\mathrm{d}}
\newcommand{\Pheat}{\mathrm P}
\newcommand{\Hheat}{\mathrm H}

\newcommand{\der}{\mathrm D}

\newcommand{\HH}{\mathrm H}
\title{Improvement of Wu's logarithmic Sobolev inequality via the Poisson-F\"ollmer process}

\author{Shrey Aryan}

\address{Department of Mathematics, Massachusetts Institute of Technology, Cambridge, MA, USA}

\email{shrey183@mit.edu}

\author{Pablo L\'{o}pez-Rivera}
\address{Laboratoire Jacques-Louis Lions (LJLL),  Universit\'{e} Paris Cit\'{e} \& Sorbonne Universit\'{e}, CNRS, Paris
F-75013, France}
\email{plopez@math.univ-paris-diderot.fr}

\author{Yair Shenfeld}
\address{Division of Applied Mathematics, Brown University, Providence, RI, USA}
\email{Yair\_Shenfeld@Brown.edu}

\begin{document}
\maketitle

\begin{abstract}
We give an alternative proof to Wu's logarithmic Sobolev inequality for the Poisson measure on the nonnegative integers using a stochastic variational formula for entropy. We show that this approach leads to improvement of Wu's inequality under convexity assumptions.
\end{abstract}

\section{Introduction}
\label{sec:intro}

The classical logarithmic Sobolev inequality for the Gaussian measure, due to Gross \cite{gross1975logarithmic} and Stam \cite{MR109101}, has numerous applications in probability, analysis, and geometry \cite{MR1767995}. On the other hand, a logarithmic Sobolev inequality cannot hold for the discrete analogue of the Gaussian, namely  the Poisson measure on the nonnegative integers \cite{MR1636948}. However, it was shown by Bobkov and Ledoux \cite{MR1636948} that \emph{modified} logarithmic Sobolev inequalities do hold for the Poisson measure, which in turn imply concentration of measure properties for Poisson random variables. The sharpest form of these inequalities is due to Wu \cite{MR1800540} (who in fact proved them for more general Poisson point processes).

\begin{theorem}[Wu's inequality \cite{MR1800540}]
\label{thm:Wu}
Fix $T>0$, and let $\pi_T$ be the Poisson measure on $\N$ with intensity $T$. Let $f:\N\to (0,\infty)$ be an $L^1(\pi_T)$-integrable function, and define
\begin{equation*}
\Ent_{\pi_T}[f]:=\sum_{k=0}^{\infty}f(k)\log f(k)\pi_T(k)-\left(\sum_{k=0}^{\infty}f(k)\pi_T(k)\right)\log\left(\sum_{k=0}^{\infty}f(k)\pi_T(k)\right).
\end{equation*}
%\label{eq:Ent_def}
Then,
\begin{equation}
\label{eq:Wu_inq_intro}
\Ent_{\pi_T}[f]\le T\sum_{k=0}^{\infty}f(k+1)\left\{\log\left(\frac{ f(k+1)}{f(k)}\right)-1+\frac{f(k)}{f(k+1)}\right\} \pi_T(k).
\end{equation}
\end{theorem}
In the Gaussian setting, there is a beautiful proof of the classical  logarithmic Sobolev inequality due to Lehec \cite{lehec2013representation},  who showed how to deduce the inequality from a stochastic representation formula for the relative entropy with respect to the Gaussian. Our first result is to show that  using a stochastic representation formula for the relative entropy with respect to the Poisson measure (due to Klartag and Lehec \cite{Klartag_Lehec} who built on the earlier work of Budhiraja, Dupuis, and Maroulas \cite{MR2841073}), we can give a precise analogue of this proof in the discrete setting to prove Wu's inequality  \eqref{eq:Wu_inq_intro}. As a by-product of this proof technique we will obtain an improvement of Wu's inequality under convexity assumptions. Concretely, as we show in Section \ref{sec:eq-wu}, equality is attained in \eqref{eq:Wu_inq_intro} if, and only if, there exist $a,b\in \R$ such that
\begin{equation}
\label{eq:equality_Wu_intro}
f(k)=e^{ak+b}\quad\text{for all } k\in \N. 
\end{equation}
Denoting by $\delta(f)$ the deficit in  \eqref{eq:Wu_inq_intro},
\begin{equation}
\label{eq:deficit_def}
\delta(f):=T\sum_{k=0}^{\infty}f(k+1)\left\{\log\left(\frac{ f(k+1)}{f(k)}\right)-1+\frac{f(k)}{f(k+1)}\right\} \pi_T(k)-\Ent_{\pi_T}[f],
\end{equation}
an improvement of Wu's inequality should lower bound $\delta(f)$ by a positive quantity, for $f$ which is not of the form \eqref{eq:equality_Wu_intro}.  In this work we show that, under convexity type assumptions on $f$, we can use the entropy representation formula to obtain a strict lower bound on $\delta(f)$. In the Gaussian setting this program was carried out by Eldan, Lehec, and Shenfeld \cite{eldan2020stability}, who obtained a number of improvements and stability results. The discrete setting raises new challenges which, as we show in this work, can nonetheless be (partially) overcome. We refer the reader to Section \ref{sec:gauss} for a discussion comparing the Gaussian and Poisson settings.

Our improvement of Wu's inequality will hold under the assumption that $f$ is \textbf{ultra-log-concave}:
\begin{equation}
\label{eq:alpha_lcvx_intro_def}
kf(k)^2\ge (k+1)f(k+1)f(k-1), \quad\text{for all }k\in \N,\quad\text{with} \quad f(-1):=0.
\end{equation}
The equality cases \eqref{eq:equality_Wu_intro} of Wu's inequality are not ultra-log-concave, so if $f$ is ultra-log-concave we should be able to lower bound $\delta(f)$ by a strictly positive quantity. Our next result provides such lower bound in terms of $\EE[\mu]$, the mean of $\mu:=f\pi_T$, and the values of $f(0),f(1)$. These parameters naturally appear in functional inequalities for ultra-log-concave measure; cf. \cite[Remark 1.6]{lopez2024poisson}.

\begin{theorem}[Improvement of  Wu's inequality under ultra-log-concavity]
\label{thm:stab_ultra_log_concave_intro}
Fix $T>0$. Let $f:\N\to (0,\infty)$ be a function which is $L^1(\pi_T)$-integrable, ultra-log-concave, and satisfies $\int f\diff \pi_T=1$. Let $\mu:=f\pi_T$. Then,
\begin{equation}
\label{eq:stab_ultra}
\delta(f)\ge \frac{T^2}{2}\,\Theta_{\frac{f(0)}{f(1)}}\left(\frac{\EE[\mu]}{T}\right),
\end{equation}
where, for $c>0$,
\[
\Theta_c(z):=\frac{z^2}{1+cz}\log\left(\frac{1}{1+cz}\right)-\frac{z^2}{1+cz}+z^2,\quad z\ge 0.
\]
\end{theorem}
The function $\Theta_c$ is nonnegative for $z\ge 0$, and in fact strictly positive for $z>0$, so the right-hand side of \eqref{eq:stab_ultra} is always strictly positive. 
\begin{remark}
We can relate $\Theta_c$ to the relative entropy between Poisson measures of different intensities via
\begin{equation}
\label{eq:Theta_rel_ent}
\Theta_c(z)=z^2\mathrm H\left(\pi_{(1+cz)^{-1}}|\pi_1\right).
\end{equation}
On the other hand, as we show in Proposition \ref{prop:deficit_identity}, the deficit $\delta(f)$ is \emph{equal} to a (random) weighted sum of relative entropies between Poisson measures of different intensities, which makes $\Theta_c$ a natural quantity in this context.
\end{remark}

\begin{remark} The question of improvement of Wu's inequality under log-concavity assumptions is delicate. For example, when $f$ is ultra-log-concave, we have that $\mu=f\pi_T$ is \emph{$\beta$-log-concave}:
\[
\frac{\mu(k+1)^2-\mu(k+2)\mu(k)}{\mu(k+1)\mu(k+2)}\ge \beta,\quad\text{for all }k\in \N,
\]
with $T\beta=\frac{f(0)}{f(1)}$ \cite[Lemma 5.1]{MR3729642}. On the other hand, in this setting, $\EE[\mu]\le \frac{1}{\beta}$ \cite[Lemma 5.3]{MR3729642}. One might wonder whether we have the following improvement of Wu's inequality (which will be stronger than \eqref{eq:stab_ultra} as $z\mapsto \Theta_c(z)$ is increasing): For all $\beta$-log-concave measures,
\begin{equation}
\label{eq:stab_conj}
\delta(f)\overset{?}{\ge}  \frac{T^2}{2}\, \Theta_{T\beta}\left(\frac{1}{T\beta}\right)=\frac{1-\log (2)}{4}\frac{1}{\beta^2}.
\end{equation}
The answer is negative: let $\mu=f\pi_T$  with $f(k)=e^{ak+b}$ as in  \eqref{eq:equality_Wu_intro}, so that $\mu$ is $\frac{1}{Te^a}$-log-concave. Then the left-hand side of \eqref{eq:stab_conj} vanishes  (since $f$ is an equality case of Wu's inequality), but the right-hand side is strictly positive, which is impossible. 
\end{remark}

\subsection*{Organization of paper}
In Section \ref{sec:proof_inq} we introduce the Poisson-F\"ollmer process.  In Section \ref{sec:stab_inq} we introduce the entropy representation formula, provide a new proof of Wu's inequality (Theorem \ref{thm:Wu}), and prove our main result Theorem \ref{thm:stab_ultra_log_concave_intro}. In Section \ref{sec:eq-wu} we characterize the equality cases of Wu's inequality \eqref{eq:Wu_inq_intro}; this section is not needed for the results in the previous sections and is included for completeness. In Section \ref{sec:gauss} we discuss the comparison between the Gaussian and  Poisson settings.

\subsection{Acknowledgments} 
We thank Max Fathi and Joseph Lehec for useful conversations. We are very grateful to the referees who provided very valuable comments that improved the manuscript. 

This project has received funding from the European Union's Horizon 2020 research and innovation programme under the Marie Sk\l{}odowska-Curie grant agreement No 945332. This work has also received support under the program ``Investissement d'Avenir" launched by the French Government and implemented by ANR, with the reference ``ANR-18-IdEx-0001" as part of its program ``Emergence".  This work received funding from the Agence Nationale de la Recherche (ANR) Grant ANR-23-CE40-0003 (Project CONVIVIALITY), as well as funding from the Institut Universitaire de France.

This material is based upon work supported by the National Science Foundation under Awards
DMS-2331920 and DMS-2508545.

\section{The Poisson-F\"ollmer process}
\label{sec:proof_inq}
In this section we define the Poisson-F\"ollmer process and prove some of its fundamental properties. We begin with general preliminaries on the Poisson semigroup.
\subsection{Preliminaries}  
We denote the natural numbers as $\N:=\{0,1,\ldots\}$, the integers as $\mathbb Z$,  and the real numbers as $\R$. For $t\ge 0$ we denote by $\pi_t$ the Poisson measure on $\N$ with intensity $t$,
\begin{equation}
\label{eq:poisson}
\pi_t(k):=e^{-t}\frac{t^k}{k!},\quad \forall~ k\in \N,
\end{equation}
where $\pi_0$ is  a point-mass at 0. We say that a function $f:\N\to \R$ is \emph{$L^1(\pi_t)$-integrable} if 
\[
\int |f|\diff \pi_t:=\sum_{n\in\N}|f(n)|\pi_t(n)<\infty.
\]
Fix $t>0$ and a function $f:\N\to\R$ which is $L^1(\pi_t)$-integrable. We define
\begin{equation}
\label{eq:Poisson_semigroup}
\Pheat_tf(k):=\sum_{n\in\N}f(k+n)\pi_t(n),\quad \forall ~k\in \N.
\end{equation}
For $t=0$ we set $\Pheat_0f=f$. The \emph{Poisson semigroup} is the collection of operators  $(\Pheat_t)_{t\ge 0}$, and it satisfies a heat equation of the form
\begin{equation}
\label{eq:heat}
\partial_t\Pheat_t f(k)=\der \Pheat_tf(k)=\Pheat_t\der f(k),\quad \forall ~k\in \N,
\end{equation}
where
\begin{equation}
\label{eq:derivative_def}
\der f(k):=f(k+1)-f(k), \quad \forall ~k\in \N.
\end{equation}
Fix $T>0$. Given $f:\N\to (0,\infty)$ which is $L^1(\pi_{T-t})$-integrable we will denote
\begin{equation}
\label{eq:Feq}
F(t,k):=\log \Pheat_{T-t}f(k),\quad\forall~k\in \N,
\end{equation}
and 
\begin{equation}
\label{eq:Geq}
G(t,k):=e^{\der F(t,k)},\quad\forall~k\in \N.
\end{equation}

\begin{lemma}
We have
\label{lem:time-der-G}
\begin{equation}
\label{eq:heat_exp}
\partial_tF(t,k)=-e^{\der F(t,k)}+1,\quad\forall~k\in \N,
\end{equation}
and 
\begin{equation}
\label{eq:heat_exp_G}
\partial_tG(t,k)=-G(t,k)\der G(t,k),\quad\forall~k\in \N.
\end{equation}
\end{lemma}
\begin{proof}
Equation \eqref{eq:heat_exp} follows from \eqref{eq:heat}. 
For equation \eqref{eq:heat_exp_G} we use \eqref{eq:Geq} and \eqref{eq:heat_exp},
\begin{align*}
\partial_tG(t,k)&=\partial_te^{\der F(t,k)}=e^{\der F(t,k)}\partial_t(\der F(t,k))=G(t,k)\der  (\partial_tF(t,k))\\
&=G(t,k)\der\left(-e^{\der F(t,k)}+1\right)=-G(t,k)\der G(t,k).
\end{align*}
\end{proof}

\subsection{The Poisson-F\"ollmer process}
\label{subsec:PF_process}
We now turn to the construction and properties of the Poisson-F\"ollmer process. Fix $T>0$,  and let  $\mu:=f\pi_T$ be a positive probability measure on $\N$. Klartag and Lehec \cite{Klartag_Lehec}, building on and specializing the work of Budhiraja, Dupuis, and Maroulas \cite{MR2841073}, constructed a stochastic counting process $(X_t)_{t\in [0,T]}$ such that $X_T\sim \mu$,  whenever $f$ is bounded or log-concave\footnote{The work \cite{Klartag_Lehec} considered the case when $f$ is bounded, and in \cite{lopez2024poisson} it was shown that boundedness can be replaced by log-concavity.}. We will describe the process briefly and refer to \cite{Klartag_Lehec, lopez2024poisson} for a complete description. We let $(\Omega,\mathcal F,\PP)$ be the underlying probability space on which the following random variables are defined. We let $N$ be a Poisson point process on $[0,T]\times (0,\infty)$ with Lebesgue intensity measure. Then $N(F)$, for a Borel set $F\subset [0,T]\times (0,\infty)$, is a Poisson random variable with intensity equal to the Lebesgue measure of $F$. For $t\in [0,T]$ we let $\mathcal F_t$ be the sigma-algebra generated by the following collection of random variables,
\[
\mathcal F_t:=\sigma\left(\{N(F): F\subset [0,t]\times (0,\infty)\text{ is a Borel set}\}\right).
\]
The collection $(\mathcal F_t)_{t\in [0,T]}$ is a filtration, and we say that a stochastic process $(\lambda_t)_{t\in [0,T]}$, where $\lambda_t:\Omega\to\R$, is \emph{predictable}, if the function $(t,\omega)\mapsto \lambda_t(\omega)$ is measurable with respect to the sigma-algebra $\sigma(\{(s,t]\times B: s\le t\le T, B\in\mathcal F_s\})$. Let $(\lambda_t)_{t\in [0,T]}$ be a predictable, nonnegative, and bounded stochastic process. Define
\begin{equation}
\label{eq:Xt_def}
X_t^{\lambda}(\omega):=N\left(\left\{(s,x)\in [0,T]\times (0,\infty): s< t,~ x\le \lambda_s(\omega)\right\}\right),\quad\text{(see Figure \ref{fig:X})}.
\end{equation}

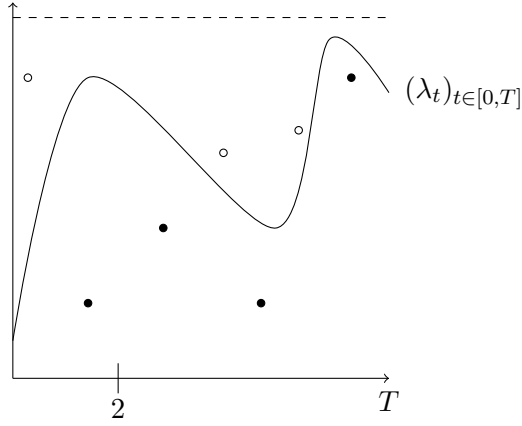
\begin{figure}
\centering
\begin{tikzpicture}[scale=1]
\draw[->] (0,0) -- (5,0);
  \draw[->] (0,0) -- (0,5);
  \draw[dashed] (0,4.8) -- (5,4.8);
  \draw (1.4,0.2) -- (1.4,-0.2);
\draw [black] plot [smooth, tension=1] coordinates { (0,.5) (1,4) (3.5,2) (4.2,4.5) (5,3.8)};
%\node at (-0.4,4.8) {$\ubd$};
\node at (5,-0.3) {$\TT$};
\node at (1.4,-0.4) {$2$};
\node at (6,3.8) {$(\ins_t)_{t\in [0,\TT]}$};
\draw[black,] (0.2,4) circle (.3ex);
\draw[black,fill=black] (1,1) circle (.3ex);
\draw[black,fill=black] (2,2) circle (.3ex);
\draw[black,] (2.8,3) circle (.3ex);
\draw[black,fill=black] (3.3,1) circle (.3ex);
\draw[black,] (3.8,3.3) circle (.3ex);
\draw[black,fill=black] (4.5,4) circle (.3ex);
\end{tikzpicture}
\caption{The points in $[0,\TT]\times \R_{\ge 0}$ are generated according to a standard Poisson process (7 points in this case). At time $t\in [0,\TT]$ the value of the process $X_t^{\lambda}$ is equal to the number of points under the curve (filled circles). In the figure $X_2^{\lambda}=1$ and $X_{\TT}^{\lambda}=4$.}
\label{fig:X}
\end{figure}

Klartag and Lehec  \cite{Klartag_Lehec} showed that we can choose a particular density $(\lambda_t)_{t\in [0,T]}$ given by
\begin{equation}
\label{eq:lambda_def}
\lambda_t:=\frac{\Pheat_{T-t} f(X_t+1)}{\Pheat_{T-t} f(X_t)},
\end{equation}
where $(\Pheat_t)$ is the Poisson semigroup, and that the resulting process $(X_t)_{t\in [0,T]}:=(X_t^{\lambda})_{t\in [0,T]}$ is well-defined. Moreover, $X_T\sim \mu$. We call this process the \emph{Poisson-F\"ollmer process}, since it is the discrete analogue of the F\"ollmer process in continuous setting; see Section \ref{sec:gauss}. Let us establish some useful properties of $(X_t)_{t\in [0,T]}$ and $(\lambda_t)_{t\in [0,T]}$. 

\begin{lemma}{\cite[Lemma 3.2]{lopez2024poisson}}
\label{lem:law_Xt}
Let $(X_t)_{t\in [0,T]}$ be the Poisson-F\"ollmer process. Then, $X_t\sim (\Pheat_{T-t}f)\pi_t$. 
\end{lemma}

\begin{lemma}
\label{lem:lambda_X_properties}
$~$

\begin{enumerate}[(1)]
\item Denote the \emph{compensated} process $(\tilde X_t)_{t\in [0,T]}$ as
\begin{equation}
\label{eq:compens_def}
\tilde X_t:=X_t-\int_0^t\lambda_s\diff s.
\end{equation}
Then $(\tilde X_t)_{t\in [0,T]}$ is a martingale.

\item The process $(\lambda_t)_{t\in [0,T]}$ is a martingale.

\item For every $t\in [0,T]$,
\begin{equation}
\label{eq:Xt_mean}
\EE[X_t]=t\,\EE[\lambda_t]=\frac{t}{T}\EE[\mu].
\end{equation}
%=t\,\Pheat_Tf(1)
\end{enumerate}
\end{lemma}
\begin{proof}
Item (1) can be found in \cite[p. 100]{Klartag_Lehec}, while item (2) can be found in \cite[Lemma 3.3]{lopez2024poisson} (see also Lemma \ref{lem:lambda_mart} below). Taking expectation in \eqref{eq:compens_def}, and using item (2), we find that $\EE[X_t]=t\,\EE[\lambda_t]$. Using $\EE[X_t]=t\,\EE[\lambda_t]$ with $t=T$ gives, since $X_T\sim \mu$, $\EE[\mu]=\EE[X_T]=T\,\EE[\lambda_T]=T\,\EE[\lambda_t]$, where the last equality holds by item (2). It follows that $\EE[\lambda_t]=\frac{\EE[\mu]}{T}$, and hence $t\EE[\lambda_t]=\frac{t}{T}\EE[\mu]$.
\end{proof}
As we saw in Lemma \ref{lem:lambda_X_properties}, the process  $(\lambda_t)_{t\in [0,T]}$ is a martingale. The next result gives an explicit expression for this martingale. To this end we recall \eqref{eq:Geq} and note that
\begin{equation}
\label{eq:Glambda_eq}
\lambda_t=G(t,X_t).
\end{equation}
We recall that a stochastic integral of the form $\int \cdot \diff X_t$ is simply the sum of the integrand at the jump points of the process $( X_t)_{t\in [0,T]}$; cf. \cite{Klartag_Lehec,  lopez2024poisson}. The  integral against the compensated process $(\tilde X_t)_{t\in [0,T]}$ (Lemma \ref{lem:lambda_X_properties}) is defined as the sum of the integrals $\int \cdot \diff X_t$ and $\int \cdot \diff t$. 
\begin{lemma}
\label{lem:lambda_mart}
We have
\begin{equation}
\label{eq:lambda_1st_var}
\diff \lambda_t=\der G(t,X_t)\diff \tilde X_t.
\end{equation}
\end{lemma}
\begin{proof}
\begin{align*}
\lambda_t&=G(t,X_t)=\int_0^t \partial_sG(s,X_s) \diff s+\int_0^t \der G(s,X_s)\diff X_s\\
&=\int_0^t [\partial_sG(s,X_s) +\der G(s,X_s)G(s,X_s)]\diff s+\int_0^t \der G(s,X_s)\diff \tilde X_s,
\end{align*}
and the first integrand vanishes by \eqref{eq:heat_exp_G}. 
\end{proof}

\section{Improvement of Wu's inequality}
\label{sec:stab_inq}
This section is dedicated to the proof of our main result, Theorem \ref{thm:stab_ultra_log_concave_intro}. We begin with Section \ref{subsec:wu_inq_proof} where we provide a new proof of Wu's inequality \eqref{eq:Wu_inq_intro} using the Poisson-F\"ollmer process. This proof is the analogue of Lehec's proof of the Gaussian logarithmic Sobolev inequality using the F\"ollmer process---see Section \ref{sec:gauss}. To prove the improvement \eqref{eq:stab_ultra} we will derive in Section \ref{subsec:iden_deficit} an identity for the deficit \eqref{eq:deficit_def} using the Poisson-F\"ollmer process. Using this identity we prove Theorem \ref{thm:stab_ultra_log_concave_intro} in Section \ref{subsec:improv}.

\subsection{Proof of Wu's inequality via the Poisson-F\"ollmer process}
\label{subsec:wu_inq_proof}

It suffices to prove Theorem \ref{thm:Wu} under the assumption that $f$ is bounded from above and below \cite[\S 1.4]{MR1800540}. In particular, as explained in Section \ref{subsec:PF_process}, this assumption guarantees the existence of the Poisson-F\"ollmer process. Our proof of Wu's inequality hinges on the following entropy representation formula  \cite[Remark 3, p. 102]{Klartag_Lehec}\footnote{In Section \ref{sec:eq-wu} we provide a proof of this result under more relaxed assumptions on $f$.}, where we use the notation
\begin{equation}
\label{eq:rel_ent_def}
\HH(\mu|\pi_T):=\sum_{k=0}^{\infty} \log\left(\frac{\mu(k)}{\pi_T(k)}\right)\mu(k),
\end{equation}
for the relative entropy of a positive measure $\mu$ on $\N$ with respect to $\pi_T$.  Note that if $\mu=f\pi_T$, then
\begin{equation}
\label{eq:relent_ent}
\HH(\mu|\pi_T)=\Ent_{\pi_T}[f].
\end{equation}

\begin{proposition}[Entropy representation formula]
\label{prop:entropy_rep}
Let $\mu=f\pi_T$ be a probability measure on $\mathbb N$  for which the Poisson-F\"ollmer process  $(X_t)_{t\in [0,T]}$ associated to $\mu$ exists\footnote{The existence of the Poisson-F\"ollmer process is guaranteed if $f$ is either bounded or log-concave; see  Section \ref{subsec:PF_process}.}. Let $(\lambda_t)_{t\in [0,T]}$ be as in \eqref{eq:lambda_def}. Then,
\label{prop:entropy_rep}
\begin{equation}
\label{eq:entropy_rep_lambda}
\HH(\mu|\pi_T)=\int_0^T\EE[\lambda_t\log\lambda_t-\lambda_t+1]\diff t.
\end{equation}
\end{proposition}
With Equation \eqref{eq:entropy_rep_lambda} in hand we can provide the following one-line proof of Theorem \ref{thm:Wu}. By Lemma \ref{lem:lambda_X_properties}(2) the process $(\lambda_t)_{t\in [0,T]}$ is a martingale, so since the function $(0,\infty)\ni r\mapsto r\log r-r+1$ is convex, we get
\begin{equation}
\label{eq:one-line-proof}
\HH(\mu|\pi_T)\overset{\eqref{eq:entropy_rep_lambda}}{=}\int_0^T\EE[\lambda_t\log\lambda_t-\lambda_t+1]\diff t\overset{\textnormal{Jensen's inq.}}{\le} T\,\EE[\lambda_T\log\lambda_T-\lambda_T+1].
\end{equation}
The right-hand side of \eqref{eq:one-line-proof} is precisely the right-hand side of \eqref{eq:Wu_inq_intro}. 

\subsection{An identity for the deficit in  Wu's inequality}
\label{subsec:iden_deficit}
We will now derive an identity for the deficit \eqref{eq:deficit_def} using the Poisson-F\"ollmer process. To state the deficit identity denote 
\begin{equation}
\label{eq:D2_def}
\der^2:=\der\circ \der,
\end{equation}
and also recall the definition \eqref{eq:Feq} of $F(t,k)$. 
\begin{proposition}
\label{prop:deficit_identity}
Let $\mu=f\pi_T$ be a probability measure on $\mathbb N$  for which the Poisson-F\"ollmer process  $(X_t)_{t\in [0,T]}$ associated to $\mu$ exist, and let $(\lambda_t)_{t\in [0,T]}$ be defined by \eqref {eq:lambda_def}. Then, 
\[
\delta(f)=\int_0^T\int_t^T\EE\left[\lambda_s^2\HH\left(\pi_{e^{\der^2F(s,X_s)}}\bigg |\pi_1\right) \right]\diff s\diff t.
\]
\end{proposition}
The proof of Proposition \ref{prop:deficit_identity} requires  a number of preliminary results. Recall that by Proposition \ref{prop:entropy_rep}, 
\[
\HH(\mu|\pi_T)=\int_0^T\EE[\phi(\lambda_t)]\diff t,
\]
where $\phi:(0,\infty)\to [0,\infty)$ is a convex function given by 
\begin{equation}
\label{eq:rate_fn_poisson}
\phi(r):=r\log r-r+1. 
\end{equation}
The next result uses Taylor expansion to re-express $\EE[\phi(\lambda_t)]$, and hence $\HH(\mu|\pi_T)$. 
\begin{lemma}
\label{lem:phi_der}
Let $G$ be as in \eqref{eq:Geq}. Then,
\[
\EE[\phi(\lambda_t)]=\int_0^t\lambda_s[\phi(\lambda_s+\der G(s,X_s))-\phi(\lambda_s)-\phi'(\lambda_s)\der G(s,X_s)]\diff s.
\]
\end{lemma}
\begin{proof}
Recall the martingale $(\tilde X_t)_{t\in [0,T]}$ from Lemma \ref{lem:lambda_X_properties}, and recall Equation \eqref{eq:Glambda_eq}. We have 
\begin{align*}
\phi(\lambda_t)&=\phi(G(t,X_t))=\int_0^t\partial_s(\phi\circ G)(s,X_s)\diff s+\int_0^t\der (\phi\circ G)(s,X_s)\diff X_s\\
&=\int_0^t\phi'(G(s,X_s))\partial_sG(s,X_s)\diff s+\int_0^t[\phi(G(s,X_s+1))-\phi(G(s,X_s))]\diff X_s\\
&\overset{\eqref{eq:heat_exp_G}}{=}\int_0^t[-\phi'(\lambda_s)\lambda_s\der G(s,X_s)]\diff s+\int_0^t[\phi(G(s,X_s)+\der G(s,X_s))-\phi(G(s,X_s))]\diff X_s\\
&=\int_0^t\lambda_s[\phi(G(s,X_s)+\der G(s,X_s))-\phi(G(s,X_s))-\phi'(\lambda_s)\der G(s,X_s)]\diff s\\
&+\int_0^t[\phi(G(s,X_s)+\der G(s,X_s))-\phi(G(s,X_s))]\diff \tilde X_s\\
&=\int_0^t\lambda_s[\phi(\lambda_s+\der G(s,X_s))-\phi(\lambda_s)-\phi'(\lambda_s)\der G(s,X_s)]\diff s\\
&+\int_0^t[\phi(G(s,X_s)+\der G(s,X_s))-\phi(G(s,X_s))]\diff \tilde X_s.
\end{align*}
Since $(\tilde X_s)$ is a martingale, taking expectation completes the proof. 
\end{proof}

We need two more results. The first is an identity for the relative entropy between Poisson measures with different intensities, whose  proof is immediate.
\begin{lemma}
\label{lem:entropy_between_Poisson}
\[
\HH(\pi_{\alpha}|\pi_{\beta})=\beta\HH\left(\pi_{\frac{\alpha}{\beta}}|\pi_1\right)=\alpha\left(\frac{\beta}{\alpha}-\log \left(\frac{\beta}{\alpha}\right)-1\right),\quad \forall~\alpha,\beta> 0.
\]
\end{lemma}
The second result shows that the integrand in Lemma \ref{lem:phi_der} can be written as relative entropy between Poisson measures. The proof is again immediate.
\begin{lemma}
\label{lem:phi_Taylor}
Let $\phi(r)=r\log r-r+1$ for $r\ge 0$. Then, for any $x,y> 0$,
\[
\phi(y)-\phi(x)-\phi'(x)(y-x)=\HH(\pi_y|\pi_x).
\]
\end{lemma}

Combining the above results we can now prove the deficit identity.

\begin{proof}[Proof of Proposition \ref{prop:deficit_identity}]

By Lemma \ref{lem:phi_der}, Lemma \ref{lem:entropy_between_Poisson}, and Lemma \ref{lem:phi_Taylor} we have
 \begin{align*}
&\partial_s\EE[\phi(\lambda_s)]\overset{\text{Lemma \ref{lem:phi_der}}}{=}\EE\left[\lambda_s\{\phi(\lambda_s+\der G(s,X_s))-\phi(\lambda_s)-\phi'(\lambda_s)\der G(s,X_s)\}\right]\\
&\overset{\text{Lemma \ref{lem:phi_Taylor}}}{=}\EE\left[\lambda_s\HH\left(\pi_{\lambda_s+\der G(s,X_s)}|\pi_{\lambda_s}\right) \right]\overset{\text{Lemma \ref{lem:entropy_between_Poisson}}}{=}\EE\left[\lambda_s^2\HH\left(\pi_{\frac{\lambda_s+\der G(s,X_s)}{\lambda_s}}\bigg |\pi_1\right) \right]\\
&=\EE\left[\lambda_s^2\HH\left(\pi_{e^{\der^2F(s,X_s)}}\bigg |\pi_1\right) \right],
\end{align*}
where the last equality used
\begin{align*}
\der G(t,k)=\frac{\Pheat_{T-t}f(k+2)}{\Pheat_{T-t}f(k+1)}-\frac{\Pheat_{T-t}f(k+1)}{\Pheat_{T-t}f(k)}=\frac{\Pheat_{T-t}f(k+2)}{\Pheat_{T-t}f(k+1)}-G(t,k),
\end{align*}
so
\begin{align*}
\frac{ G(t,k)+\der G(t,k)}{G(t,k)}=\frac{\Pheat_{T-t}f(k+2)\Pheat_{T-t}f(k)}{\Pheat_{T-t}f(k+1)^2}=e^{\der ^2\log \Pheat_{T-t}f(k)}.
\end{align*}
Finally, the proof is complete since by Proposition \ref{prop:entropy_rep} and the proof of Theorem \ref{thm:Wu}, 
\begin{align}
\label{eq:delta_phi}
\delta(f)=\int_0^T\EE[\phi(\lambda_T)-\phi(\lambda_t)]\diff t=\int_0^T\int_t^T\partial_s\EE[\phi(\lambda_s)]\diff s\diff t.
\end{align}
\end{proof}
Proposition \ref{prop:deficit_identity} is our main tool to obtain an improvement of Wu's inequality by almost-surely lower bounding  $\HH\left(\pi_{e^{\der^2F(s,X_s)}} |\pi_1\right)$. Specifically, Lemma \ref{lem:entropy_between_Poisson} shows that the map $\alpha\mapsto \HH(\pi_{\alpha}|\pi_1)$ is  decreasing on $(0,1]$, so to lower bound $\HH\left(\pi_{e^{\der^2F(s,X_s)}} |\pi_1\right)$ we can show that $e^{\der^2F(s,X_s)}\le 1$, and then upper bound  $e^{\der^2F(s,X_s)}$. We will show this can indeed be done when $f$ is ultra-log-concave.

\subsection{Improvement of Wu's inequality under ultra-log-concavity}
\label{subsec:improv}
In this section we prove Theorem \ref{thm:stab_ultra_log_concave_intro}.
Our first task is to establish the preservation of ultra-log-concavity under the Poisson semigroup.
\begin{lemma}
\label{lem:heat_flow_ulc}
Let $f:\N\to (0,\infty)$ be an ultra-log-concave function. Then, for every $t\ge 0$, $\Pheat_tf:\N\to (0,\infty)$ is also ultra-log-concave.
\end{lemma}
\begin{proof}
The lemma follows from the closability under convolutions of ultra-log-concave functions \cite{MR494391}: If $\{a_k\}_{k\in \mathbb Z}, \{b_k\}_{k\in \mathbb Z}$ are ultra-log-concave then $\{(a\ast b)_k\}_{k\in \mathbb Z}$ is also  ultra-log-concave, where
\begin{equation}
\label{eq:ulc_convolv}
(a\ast b)_k:=\sum_{n=-\infty}^{+\infty}a_nb_{k-n}=\sum_{n=-\infty}^{+\infty}a_{k-n}b_n.
\end{equation}
To apply \eqref{eq:ulc_convolv} in our setting we extend $f$ and $\pi_t$ to $\mathbb Z$ by setting them to zero on $\{k\in \mathbb Z:k<0\}$, and note that $f$ and $\pi_t$ remain ultra-log-concave as functions on $\mathbb Z$. Now let $\tilde{\pi}_t$ be given by $\tilde{\pi}_t(k):=\pi_t(-k)$ for $k\in \mathbb Z$, and note that $\tilde{\pi}_t$ is ultra-log-concave (as $\pi_t$ yields equality in \eqref{eq:alpha_lcvx_intro_def}). The proof is complete as $\Pheat_t f=(f\ast \tilde{\pi}_t)$. 
\end{proof}

The next result gives a useful bound for ultra-log-concave functions.
\begin{lemma}
\label{lem:ulc_prop}
Let $f:\N\to (0,\infty)$ be an ultra-log-concave function. Then, for every $k\in\N$,
\[
\frac{f(k+2)f(k)}{f(k+1)^2}\le \frac{1}{1+c\frac{f(k+1)}{f(k)}}<1\quad\text{with}\quad c:=\frac{f(0)}{f(1)}.
\]
\end{lemma}

\begin{proof}
By \cite[Lemma 5.1]{MR3729642} the fact that $f$ is ultra-log-concave means that $f$ is \emph{$c$-log-concave} with $c:=\frac{f(0)}{f(1)}$:
\begin{equation}
\label{eq:c-log-concave}
\frac{f(k+1)^2-f(k+2)f(k)}{f(k+1)f(k+2)}\ge c.
\end{equation}
Multiplying both sides of \eqref{eq:c-log-concave} by $\frac{f(k+1)}{f(k)}$, and rearranging, complete the proof.
\end{proof}

We are now ready for the proof of Theorem \ref{thm:stab_ultra_log_concave_intro}.  The function $\Pheat_{T-t}f$ is ultra-log-concave by  Lemma \ref{lem:heat_flow_ulc}, so Lemma \ref{lem:ulc_prop} can be applied to give
\begin{equation}
\label{eq:expDD_ulc}
e^{\der^2F(s,X_s)}=\frac{\Pheat_{T-s}f(X_s+2)\Pheat_{T-s}f(X_s)}{(\Pheat_{T-s}f)^2(X_s+1)}\le \frac{1}{1+c_s\frac{\Pheat_{T-s}f(X_s+1)}{\Pheat_{T-s}f(X_s)}}=\frac{1}{1+c_s\lambda_s},
\end{equation}
with $c_s:=\frac{\Pheat_{T-s}f(0)}{\Pheat_{T-s}f(1)}$. Next we will lower bound $c_s$. We claim that
\begin{equation}
\label{eq:ct_bd}
c_s=\frac{\Pheat_{T-s}f(0)}{\Pheat_{T-s}f(1)}\ge \frac{f(0)}{f(1)}.
\end{equation}
Indeed, it suffices to show that the function $[0,T]\ni s\mapsto \eta(s):=\frac{\Pheat_{T-s}f(0)}{\Pheat_{T-s}f(1)}$ is non-increasing since the right-hand side of \eqref{eq:ct_bd} is equal to $\eta(T)$. The latter holds since, using \eqref{eq:heat}, we have 
\[
\partial_s\eta(s)=-\frac{1}{(\Pheat_{T-s}f)^2(1)}\left\{(\Pheat_{T-s}f)^2(1)-\Pheat_{T-s}f(2)\Pheat_{T-s}f(0)\right\}\le 0,
\]
where the inequality holds as $\Pheat_{T-s}f$ is ultra-log-concave (Lemma \ref{lem:heat_flow_ulc}). Combining \eqref{eq:expDD_ulc} and \eqref{eq:ct_bd} we conclude that 
\begin{equation}
\label{eq:expDD_ulc_final}
e^{\der^2F(s,X_s)}\le \frac{1}{1+\frac{f(0)}{f(1)}\lambda_s}<1.
\end{equation}
Since $\alpha\mapsto \HH(\pi_{\alpha}|\pi_1)$ is decreasing on $(0,1]$, we get from \eqref{eq:expDD_ulc_final} that, almost-surely,
\begin{equation}
\label{eq:rel_ent_estimate_stochastic}
\HH\left(\pi_{e^{\der^2F(s,X_s)}}\bigg |\pi_1\right) \ge \HH\left(\pi_{(1+\frac{f(0)}{f(1)}\lambda_s)^{-1}}\bigg |\pi_1\right).
\end{equation}
It follows from Proposition \ref{prop:deficit_identity} that
\begin{align*}
\delta(f)&\ge \int_0^T\int_t^T\EE\left[\lambda_s^2\HH\left(\pi_{(1+\frac{f(0)}{f(1)}\lambda_s)^{-1}}\bigg |\pi_1\right) \right]\diff s\diff t= \int_0^T\int_t^T\EE\left[\Theta_{\frac{f(0)}{f(1)}}(\lambda_s) \right]\diff s\diff t,
\end{align*}
where the last equality follows from the definition of $\Theta_c$ and Lemma \ref{lem:entropy_between_Poisson}. The function $z\mapsto \Theta_c(z)$ can be verified to be convex, 
so by Jensen's inequality,
\begin{align}
\label{eq:Jensen_Phi}
\delta(f)\ge  \int_0^T\int_t^T\Theta_{\frac{f(0)}{f(1)}}\left(\EE[\lambda_s]\right)\diff s\diff t\overset{\eqref{eq:Xt_mean}}{=}\int_0^T\int_t^T\Theta_{\frac{f(0)}{f(1)}}\left(\frac{\EE[\mu]}{T}\right)\diff s\diff t=\frac{T^2}{2}\,\Theta_{\frac{f(0)}{f(1)}}\left(\frac{\EE[\mu]}{T}\right).
\end{align}

\section{The equality cases of Wu's inequality}
\label{sec:eq-wu}
In this section we characterize the equality cases of Wu's inequality \eqref{eq:Wu_inq_intro}. This characterization is not needed for our proof of \eqref{eq:Wu_inq_intro} via the Poisson-F\"ollmer process, or for the proof of Theorem \ref{thm:stab_ultra_log_concave_intro}, but we could not find this result in the literature so we add it for completeness. The first step towards this characterization is to get an entropy representation formula for $\Ent_{\pi_T}[f]$ which holds for a broad class of functions. Let us set up some notation first. We define
\begin{equation}
\label{eq:Phi_def}
\Phi(r):=r\log r,\quad \quad r\ge 0,
\end{equation}
and define the function $\Psi(u,v)$, for $u>0$ and $u+v>0$, by
\begin{equation}
\label{eq:Psi_def}
\begin{split}
\Psi(u,v)&:=\Phi(u+v)-\Phi(u)-\Phi'(u)v\\
&=(u+v)\log(u+v)-u\log u-(\log u+1)v.
\end{split}
\end{equation}
Note that $\Psi$ is nonnegative and convex when $u>0$ and $u+v>0$  \cite[\S 1.4]{MR1800540}. With the above notation, Wu's inequality \eqref{eq:Wu_inq_intro} can be written as,
\begin{equation}
\label{eq:Wu_inq_Psi} 
\Ent_{\pi_T}[f]\le T\,\EE_{\pi_T}[\Psi(f,\der f)]\quad\textnormal{ for all $L^1(\pi_T)$-integrable functions $f:\mathbb N\to (0,\infty)$}.
\end{equation}
We can now state the entropy representation formula. 
\begin{proposition}
\label{prop:entropy_rep_semi}
Let $f:\N\to (0,\infty)$ be an $L^1(\pi_T)$-integrable function such that $\Ent_{\pi_T}[f]<\infty$ and $\EE_{\pi_T}[\Psi(f,\der f)]<\infty$. Then,
\begin{equation}
\label{eq:entropy_rep_semigroup}
\Ent_{\pi_T}[f]=\int_0^T\EE_{\pi_t}[\Psi(\Pheat_{T-t}f,\der\Pheat_{T-t}f)]\diff t.
\end{equation}
\end{proposition}
Under more restrictive assumptions on $f$, the representation formula \eqref{eq:entropy_rep_semigroup}  appears as a special case of \cite[Equation (1.7)]{MR1800540}, and also follows from the variational formula  in \cite[Remark 3, p. 102]{Klartag_Lehec}; cf. Proposition \ref{prop:entropy_rep}. But for the purpose of characterizing the equality cases of Wu's inequality (for the same class of functions for which the inequality holds), we need \eqref{eq:entropy_rep_semigroup} to hold for all functions $f$ satisfying the assumptions of Proposition \ref{prop:entropy_rep_semi}. In order to avoid dealing with the technical question regarding the most general conditions under which the Poisson-F\"ollmer process exist, we will use semigroup tools.

\begin{proof}[Proof of Proposition \ref{prop:entropy_rep_semi}] 
First note that since $\Psi$ is convex \cite[\S 1.4]{MR1800540}, Jensen's inequality yields 
\[
\Psi(\Pheat_{T-t}f,\der\Pheat_{T-t}f)\le \Pheat_{T-t}\Psi(f,\der f),
\]
and hence the integrand on the right-hand side of \eqref{eq:entropy_rep_semigroup} is always finite,
\begin{equation}
\label{eq:Psi_jensen}
\EE_{\pi_t}[\Psi(\Pheat_{T-t}f,\der\Pheat_{T-t}f)]\le \EE_{\pi_t}[\Pheat_{T-t}\Psi(f,\der f)]=\EE_{\pi_T}[\Psi(f,\der f)]<\infty.
\end{equation}
Fix $k\in \N$ and define $\alpha:[0,T]\to\R$ by
\begin{equation}
\label{eq:alpha_def}
\alpha(t):=\EE_{\pi_t}[\Pheat_{T-t}f\log \Pheat_{T-t}f]-\EE_{\pi_T}[f]\log(\EE_{\pi_T}[f]),
\end{equation}
and note that the convexity of $r\mapsto r\log r$, and Jensen's inequality, yield
\[
\EE_{\pi_t}[\Pheat_{T-t}f\log \Pheat_{T-t}f]\le \EE_{\pi_t}\Pheat_{T-t}[f\log f]=\Ent_{\pi_T}[f]+\EE_{\pi_T}[f]\log(\EE_{\pi_T}[f])<\infty,
\]
so $\alpha(t)$ is finite. Since $\pi_0=\delta_0$ and $\Pheat_Tf(0)=\EE_{\pi_T}[f]$, we have 
\begin{equation}
\label{eq:entropy_alpha}
\Ent_{\pi_T}[f]=\int_0^T \partial_t\alpha(t)\diff t.
\end{equation}
Using \eqref{eq:heat}, \eqref{eq:heat_exp}, and $\partial_t\pi_t(k)=\pi_t(k-1)-\pi_t(k)$, with $\pi_t(-1):=0$, we get
\begin{align}
\begin{split}
\partial_t\alpha(t)&=\sum_{k=0}^{\infty}[\Pheat_{T-t}f(k)-\Pheat_{T-t}f(k+1)]\pi_t(k)\label{eq:alpha_time-der}\\
&\quad +\sum_{k=0}^{\infty}\log \Pheat_{T-t}f(k)[\Pheat_{T-t}f(k)-\Pheat_{T-t}f(k+1)]\pi_t(k)\\
&\quad +\sum_{k=0}^{\infty}\Pheat_{T-t}f(k)\log \Pheat_{T-t}f(k)[\pi_t(k-1)-\pi_t(k)].
\end{split}
\end{align}
The second and third terms on the right-hand side of equation \eqref{eq:alpha_time-der} can be written as, using $\pi_t(-1)=0$,
\begin{align*}
&\sum_{k=0}^{\infty}\log \Pheat_{T-t}f(k)[\Pheat_{T-t}f(k)-\Pheat_{T-t}f(k+1)]\pi_t(k)\\
&\quad +\sum_{k=0}^{\infty}\Pheat_{T-t}f(k)\log \Pheat_{T-t}f(k)[\pi_t(k-1)-\pi_t(k)]\\
&=-\sum_{k=0}^{\infty}\Pheat_{T-t}f(k+1)\log \Pheat_{T-t}f(k)\pi_t(k)+\sum_{k=1}^{\infty}\Pheat_{T-t}f(k)\log \Pheat_{T-t}f(k)\pi_t(k-1)\\
&=-\sum_{k=0}^{\infty}\Pheat_{T-t}f(k+1)\log \Pheat_{T-t}f(k)\pi_t(k)+\sum_{k=0}^{\infty}\Pheat_{T-t}f(k+1)\log \Pheat_{T-t}f(k+1)\pi_t(k),
\end{align*}
so
\begin{align}
\begin{split}
&\partial_t\alpha(t)=\\
&\sum_{k=0}^{\infty}\left\{\Pheat_{T-t}f(k+1)\log \Pheat_{T-t}f(k+1)-\Pheat_{T-t}f(k+1)\log \Pheat_{T-t}f(k)-\der\Pheat_{T-t}f(k)\right\}\pi_t(k)\label{eq:alpha_time-der_update}\\
&=\EE_{\pi_t}[\Psi(\Pheat_{T-t}f,\der\Pheat_{T-t}f)].
\end{split}
\end{align}
\end{proof}

Using \eqref{eq:entropy_rep_semigroup}  we now turn to the characterization of the equality cases of Wu's inequality.
\begin{proposition}[Equality cases of Wu's inequality]
\label{prop:Wu_inq_proof_rep}
Fix $T>0$, and let $\pi_T$ be the Poisson measure on $\N$ with intensity $T$. Let $f:\N\to (0,\infty)$ be an $L^1(\pi_T)$-integrable function such that $\EE_{\pi_T}[\Psi(f,\der f)]<\infty$. Then, equality holds in Wu's inequality
\begin{equation}
\label{eq:Wu_inq_Psi_via_rep} 
\Ent_{\pi_T}[f]\le T\,\EE_{\pi_T}[\Psi(f,\der f)],
\end{equation}
if, and only if, $f(k)=e^{ak+b}$ for some $a,b\in \R$. 
\end{proposition}
\begin{proof}
 By \eqref{eq:entropy_rep_semigroup}, and arguing as in \eqref{eq:Psi_jensen}, we get
\begin{equation}
\label{eq:Wu_inq_Psi_via_rep_proof} 
\Ent_{\pi_T}[f]=\int_0^T\EE_{\pi_t}[\Psi(\Pheat_{T-t}f,\der\Pheat_{T-t}f)]\diff t\le  T\,\EE_{\pi_T}[\Psi(f,\der f)].
\end{equation}
If $f(k)=e^{ak+b}$ for some $a,b\in \R$ then one can verify that equality holds in \eqref{eq:Wu_inq_Psi_via_rep}. For the reverse direction, we note that equality in \eqref{eq:Wu_inq_Psi_via_rep_proof}  implies that $ t\mapsto \EE_{\pi_t}[\Psi(\Pheat_{T-t}f,\der\Pheat_{T-t}f)]$ is a constant, and hence 
\begin{equation}
\label{eq:Jensen_eq}
 \EE_{\pi_T}[\Psi(f,\der f)]=\EE_{\pi_0}[\Psi(\Pheat_{T}f,\der\Pheat_{T}f)]=\Psi(\Pheat_{T}f,\der\Pheat_{T}f)(0)=\Psi( \EE_{\pi_T}[(f,\der f)]).
\end{equation}
Using the relation
\begin{equation}
\label{eq:Psi-phi}
\Psi(u,v)=u\phi\left(\frac{u+v}{u}\right),\qquad\textnormal{for all }u>0\textnormal{ and }u+v>0,
\end{equation}
with
\[
\phi(r):=r\log r-r+1,
\]
we have
\begin{align*}
 \EE_{\pi_T}[\Psi(f,\der f)]=\sum_{k=0}^{\infty}\phi\left(\frac{f(k+1)}{f(k)}\right)f(k)\pi_T(k).
\end{align*}
On the other hand,
\begin{align*}
\Psi\left( \EE_{\pi_T}[(f,\der f)]\right)&=\EE_{\pi_T}[f]\phi\left(\frac{\EE_{\pi_T}[\der f]+\EE_{\pi_T}[f]}{\EE_{\pi_T}[f]}\right)=\EE_{\pi_T}[f]\phi\left(\frac{\sum_{k=0}^{\infty}f(k+1)\pi_T(k)}{\EE_{\pi_T}[f]}\right)\\
&=\EE_{\pi_T}[f]\phi\left(\sum_{k=0}^{\infty}\frac{f(k+1)}{f(k)}\frac{f(k)\pi_T(k)}{\EE_{\pi_T}[f]}\right),
\end{align*}
so \eqref{eq:Jensen_eq} reads
\begin{equation}
\label{eq:eq_in_jensen}
\sum_{k=0}^{\infty}\phi\left(\frac{f(k+1)}{f(k)}\right)\frac{f(k)\pi_T(k)}{\EE_{\pi_T}[f]}=\phi\left(\sum_{k=0}^{\infty}\frac{f(k+1)}{f(k)}\frac{f(k)\pi_T(k)}{\EE_{\pi_T}[f]}\right).
\end{equation}
Since $\phi$ is strictly convex on $(0,\infty)$, and as $\mathbb N\ni k\mapsto \frac{f(k)\pi_T(k)}{\EE_{\pi_T}[f]}$ is a probability measure, the equality cases of Jensen's inequality, together with \eqref{eq:eq_in_jensen}, imply that there exists a constant $c$ such that 
\begin{equation}
\label{eq:f_constnat}
f(k+1)=cf(k)\quad\forall~k\in \N,
\end{equation}
which shows that $f(k)=e^{ak+b}$ for some $a,b\in \R$. 
\end{proof}

\section{Comparison with the Gaussian setting}
\label{sec:gauss}
The analogue of Wu's inequality in the Gaussian setting is the \emph{logarithmic Sobolev inequality}. Let $\gamma$ be the standard Gaussian measure on $\R^n$. Then, the logarithmic Sobolev inequality states that, for any $\mu=f\gamma$ (for which the quantities below are well-defined),
\begin{equation}
\label{eq:LSI}
\HH(\mu|\gamma)\le\frac{1}{2} \int_{\R^n} |\nabla\log f|^2 \diff \mu.
\end{equation}
Let us now present Lehec's proof  \cite{lehec2013representation} of \eqref{eq:LSI} using the Gaussian analogue of the entropy representation formula \eqref{eq:entropy_rep_lambda}. In the continuous setting (taking $T=1$ for simplicity), the Poisson-F\"{o}llmer process is replaced by the \emph{F\"ollmer process} \cite{follmer2005entropy, follmer2006time, lehec2013representation}, which is the solution of the following stochastic differential equation,
\begin{equation}
\label{eq:Follmer}
\diff Y_t=v_t\diff t+\diff B_t, \quad Y_0=0,
\end{equation}
where $(B_t)_{t\ge 0}$ is a standard Brownian motion in $\R^n$, and
\begin{equation}
\label{eq:v_def}
v_t:=\nabla\log \Hheat_{1-t}f(Y_t),
\end{equation}
with $(\Hheat_t)$ the heat semigroup,
\begin{equation}
\label{eq:heat_semigroup}
\Hheat_t f(x):=\int_{\R^n} f(x+\sqrt{t}z)\diff \gamma(z). 
\end{equation}
The process $(Y_t)_{t\in [0,1]}$ satisfies $Y_1\sim \mu=f\pi$, and we have the entropy representation formula
\begin{equation}
\label{eq:entropy_rep_Gauss}
\HH(\mu|\gamma)=\frac{1}{2}\int_0^1 \EE[\varphi(v_t)] \diff t,
\end{equation}
where
\begin{equation}
\label{eq:rate_fn_gauss}
\varphi(x):=\frac{x^2}{2}. 
\end{equation}
The representation \eqref{eq:entropy_rep_Gauss} is the analogue of \eqref{eq:entropy_rep_lambda}. In Table \ref{tab:comparison} we summarize the comparisons of the  stochastic constructions in the Poisson and Gaussian settings.

\begin{table}[h]
\centering
\begin{tabular}{|l|c|c|}
\hline
\multicolumn{1}{|c|}{} & \textbf{Poisson} & \textbf{Gaussian} \\
\hline
\textbf{Process} & $X_t$ \eqref{eq:Xt_def} & $Y_t$ \eqref{eq:Follmer}  \\
\textbf{Control} & $\lambda_t$ \eqref{eq:lambda_def} & $v_t$ \eqref{eq:v_def} \\
\textbf{Semigroup} & $\Pheat_t$ \eqref{eq:Poisson_semigroup} & $\Hheat_t$ \eqref{eq:heat_semigroup} \\
\textbf{Rate function} & $\phi$ \eqref{eq:rate_fn_poisson}  & $\varphi$ \eqref{eq:rate_fn_gauss} \\
\textbf{Entropy representation formula} & \eqref{eq:entropy_rep_lambda} & \eqref{eq:entropy_rep_Gauss}\\
\hline
\end{tabular}
\caption{Comparison between Poisson and  Gaussian}
\label{tab:comparison}
\end{table}
Let us now present Lehec's proof of \eqref{eq:LSI} using \eqref{eq:entropy_rep_Gauss}. It can be shown that the process $(v_t)$ is a martingale, so since $\varphi$ is convex, it follows from Jensen's inequality that
\begin{equation}
\label{eq:entropy_rep_Gauss_inq}
\HH(\mu|\gamma)=\frac{1}{2}\int_0^1 \EE[\varphi(v_t)] \diff t\le \frac{1}{2} \EE[\varphi(v_1)]=\frac{1}{2} \int_{\R^n} |\nabla\log f|^2 \diff \mu.
\end{equation}
Thus, we see that our proof of Wu's inequality is the exact discrete analogue of Lehec's proof of the Gaussian logarithmic Sobolev inequality. Turning to the question of improvement of the inequality, Eldan, Lehec, and Shenfeld \cite{eldan2020stability} used \eqref{eq:entropy_rep_Gauss} to get stability estimates for the  Gaussian logarithmic Sobolev inequality. Denote the deficit in the Gaussian logarithmic Sobolev inequality as
\[
\delta(f):=\frac{1}{2} \int_{\R^n} |\nabla\log f|^2 \diff \mu-\HH(\mu|\gamma).
\]
Then, the analogue of \eqref{eq:delta_phi} is the identity
\cite[Proposition 10]{eldan2020stability},
\begin{equation}
\label{eq:deficit_identity_LSI}
\delta(f):=\frac{1}{2}\int_0^1\EE[|v_1-v_t|^2]\diff t.
\end{equation}
However, from this point on the analogies begin to break. In \cite{eldan2020stability} it is shown that 
\begin{equation}
\label{eq:vt_der}
\EE[\varphi(v_t)]\ge \int_0^t\EE[\varphi(v_s)]^2\diff s,
\end{equation}
so differentiating \eqref{eq:vt_der} (more precisely a matrix version of this inequality) yields a differential inequality for $t\mapsto \EE[\varphi(v_t)]$, which is a key point in some of the stability estimates of \cite{eldan2020stability}. On a high-level we can view \eqref{eq:vt_der} as the analogue of Lemma \ref{lem:phi_der}. However, while in \eqref{eq:vt_der} the expression $\EE[\varphi(v_t)]$ appears on both sides of the inequality, in our setting we must deal with discrete derivatives which hinders such differential inequalities. Instead, we make crucial use of the observation in Lemma \ref{lem:phi_Taylor} that we can express the right-hand side in Lemma \ref{lem:phi_der} in terms of relative entropy, which then leads to our improvement \eqref{eq:stab_ultra}.

\bibliographystyle{amsplain0}
\bibliography{ref_Poisson_transport}

\providecommand{\bysame}{\leavevmode\hbox to3em{\hrulefill}\thinspace}
\providecommand{\MR}{\relax\ifhmode\unskip\space\fi MR }
\providecommand{\MRhref}[2]{%
  \href{http://www.ams.org/mathscinet-getitem?mr=#1}{#2}
}
\providecommand{\href}[2]{#2}
\begin{thebibliography}{10}

\bibitem{MR1636948}
S.~G. Bobkov and M.~Ledoux, \emph{On modified logarithmic {S}obolev
  inequalities for {B}ernoulli and {P}oisson measures}, J. Funct. Anal.
  \textbf{156} (1998), 347--365.

\bibitem{MR2841073}
Amarjit Budhiraja, Paul Dupuis, and Vasileios Maroulas, \emph{Variational
  representations for continuous time processes}, Ann. Inst. Henri Poincar\'{e}
  Probab. Stat. \textbf{47} (2011), 725--747.

\bibitem{eldan2020stability}
Ronen Eldan, Joseph Lehec, and Yair Shenfeld, \emph{Stability of the
  logarithmic sobolev inequality via the {F}{\"o}llmer process}, Annales de
  l'IHP Probabilit{\'e}s et statistiques, vol.~56, 2020, pp.~2253--2269.

\bibitem{follmer2005entropy}
Hans F{\"o}llmer, \emph{An entropy approach to the time reversal of diffusion
  processes}, Stochastic Differential Systems Filtering and Control:
  Proceedings of the IFIP-WG 7/1 Working Conference Marseille-Luminy, France,
  March 12--17, 1984, Springer, 2005, pp.~156--163.

\bibitem{follmer2006time}
Hans F{\"o}llmer, \emph{Time reversal on wiener space}, Stochastic
  Processes---Mathematics and Physics: Proceedings of the 1st BiBoS-Symposium
  held in Bielefeld, West Germany, September 10--15, 1984, Springer, 2006,
  pp.~119--129.

\bibitem{gross1975logarithmic}
Leonard Gross, \emph{Logarithmic {S}obolev inequalities}, American Journal of
  Mathematics \textbf{97} (1975), 1061--1083.

\bibitem{MR3729642}
Oliver Johnson, \emph{A discrete log-{S}obolev inequality under a
  {B}akry-\'{E}mery type condition}, Ann. Inst. Henri Poincar\'{e} Probab.
  Stat. \textbf{53} (2017), 1952--1970.

\bibitem{Klartag_Lehec}
Bo'az Klartag and Joseph Lehec, \emph{Poisson processes and a log-concave
  {B}ernstein theorem}, Studia Math. \textbf{247} (2019), 85--107.

\bibitem{MR1767995}
Michel Ledoux, \emph{Concentration of measure and logarithmic {S}obolev
  inequalities}, S\'eminaire de {P}robabilit\'es, {XXXIII}, Lecture Notes in
  Math., vol. 1709, Springer, Berlin, 1999, pp.~120--216.

\bibitem{lehec2013representation}
Joseph Lehec, \emph{Representation formula for the entropy and functional
  inequalities}, Annales de l'IHP Probabilit{\'e}s et statistiques, vol.~49,
  2013, pp.~885--899.

\bibitem{lopez2024poisson}
Pablo L{\'o}pez-Rivera and Yair Shenfeld, \emph{The {P}oisson transport map},
  arXiv preprint arXiv:2407.02359 (2024).

\bibitem{MR109101}
A.~J. Stam, \emph{Some inequalities satisfied by the quantities of information
  of {F}isher and {S}hannon}, Information and Control \textbf{2} (1959),
  101--112.

\bibitem{MR494391}
David~W. Walkup, \emph{P\'olya sequences, binomial convolution and the union of
  random sets}, J. Appl. Probability \textbf{13} (1976), 76--85.

\bibitem{MR1800540}
Liming Wu, \emph{A new modified logarithmic {S}obolev inequality for {P}oisson
  point processes and several applications}, Probab. Theory Related Fields
  \textbf{118} (2000), 427--438.

\end{thebibliography}

\end{document}